\documentclass[reqno,10pt,centertags]{amsart}
\usepackage{amsmath,amsthm,amscd,amssymb,latexsym,enumerate}
\usepackage{upref,esint,float,color}
\usepackage{centernot}
\usepackage{url} 

\usepackage{hyperref}

\newcommand*{\mailto}[1]{\href{mailto:#1}{\nolinkurl{#1}}}
\newcommand{\arxiv}[1]{\href{http://arxiv.org/abs/#1}{arXiv: #1}}

\makeatletter
\def\theequation{\@arabic\c@equation}

\newcommand{\bbN}{{\mathbb{N}}}
\newcommand{\bbR}{{\mathbb{R}}}

\newcommand{\bbZ}{{\mathbb{Z}}}
\newcommand{\bbC}{{\mathbb{C}}}

\newcommand{\cB}{{\mathcal B}}

\newcommand{\cH}{{\mathcal H}}

\newcommand{\no}{\nonumber}
\newcommand{\lb}{\label}
\newcommand{\bi}{\bibitem}
\newcommand{\f}{\frac}

\newcommand{\Oh}{O}

\newcommand{\dom}{\operatorname{dom}}

\renewcommand{\Re}{\operatorname{Re}}

\renewcommand{\ln}{\operatorname{ln}}
\newcommand{\tr}{\operatorname{tr}}
\DeclareMathOperator{\arcsinh}{arcsinh}
\DeclareMathOperator{\arcth}{arctanh}

\numberwithin{equation}{section}

\newtheorem{theorem}{Theorem}[section]
\newtheorem{lemma}[theorem]{Lemma}

\newtheorem{hypothesis}[theorem]{Hypothesis}

\theoremstyle{definition}

\newtheorem{remark}[theorem]{Remark}

\begin{document}

\title[Trace Formulas and the Riemann $\zeta$-Function]{Trace Formulas Applied to \\ the Riemann
$\zeta$-Function}

\author[M.\ S.\ Ashbaugh et al.]{Mark\ S.\ Ashbaugh, Fritz Gesztesy, Lotfi Hermi, Klaus Kirsten, \\ Lance Littlejohn,
and Hagop Tossounian}
\address{Department of Mathematics, University of
	Missouri, Columbia, MO 65211, USA}
\email{\mailto{ashbaughm@missouri.edu}}
\urladdr{\url{https://www.math.missouri.edu/people/ashbaugh}}

\address{Department of Mathematics,
Baylor University, One Bear Place \#97328,
Waco, TX 76798-7328, USA}
\email{\mailto{Fritz\_Gesztesy@baylor.edu}}
\urladdr{\url{http://www.baylor.edu/math/index.php?id=935340}}

\address{Department of Mathematics and Statistics, Florida International University,
11200 S.W. 8th Street, Miami, Florida 33199, USA}
\email{\mailto{lhermi@fiu.edu}}

\address{GCAP-CASPER, Department of Mathematics,
Baylor University, One Bear Place \#97328,
Waco, TX 76798-7328, USA}
\email{\mailto{Klaus\_Kirsten@baylor.edu}}
\urladdr{\url{http://www.baylor.edu/math/index.php?id=54012}}

\address{Department of Mathematics,
Baylor University, One Bear Place \#97328,
Waco, TX 76798-7328, USA}
\email{\mailto{Lance\_Littlejon@baylor.edu}}
\urladdr{\url{http://www.baylor.edu/math/index.php?id=53980}}

\address{Department of Mathematics,
Baylor University, One Bear Place \#97328,
Waco, TX 76798-7328, USA}
\email{\mailto{Hagop.Tossounian@gmail.com}}

\date{\today}
\thanks{K.K.\ was supported by the Baylor University Summer Sabbatical and Research Leave Program.}
\thanks{To appear in {\it Integrability, Supersymmetry and Coherent States. 
A volume in honour of Professor V\'eronique Hussin}, S.\ Kuru, J.\ Negro, and L.\ M.\ Nieto (eds.), CRM Series in Mathematical Physics, Springer.}

\begin{abstract}
We use a spectral theory perspective to reconsider properties of the Riemann zeta function.
In particular, new integral representations are derived and used to present its value
at odd positive integers.
\end{abstract}

\maketitle


\section{Introduction} \lb{s1}

Spectral zeta functions associated with eigenvalue problems of (partial) differential operators are of 
relevance in a wide array
of topics \cite{BKMM09, BCVZ96, Co85, Di55, Di61, DC76, EKP97, Ha77, El12, Ki02, Mi01, RS71, SC01}. 
As an example consider the Dirichlet boundary value problem
\begin{equation}
 - \Delta_D f =-f^{\prime\prime},  \quad  f(0) =f(1)=0,       \lb{1.1}
\end{equation}
where $- \Delta_D$ denotes the Dirichlet Laplacian in the Hilbert space $L^2((0,1); dx)$ (cf.\ \eqref{2.9a}), 
with purely discrete and simple spectrum, 
\begin{equation}
\sigma(- \Delta_D) = \{\lambda_k = (k\pi)^2\big\}_{k\in\bbN}.   \lb{1.3} 
\end{equation}
In particular, the spectral zeta function associated with $- \Delta_D$, 
\begin{equation}
\zeta (z; - \Delta_D) = \sum_{k \in \bbN} \lambda_k^{-z} = \pi^{-2 z} \zeta (2z), \quad \Re(z) > 1/2, 
\end{equation} 
is basically given by the Riemann zeta function
\begin{equation}
\zeta (z) = \sum_{k \in \bbN} k^{- z}, \quad \Re (z) >1.
\end{equation}
This elementary and well-known observation identifies the Riemann zeta function as a spectral zeta function and hence spectral theoretic techniques for their analysis can be applied to it. This is the perspective taken in this article. In Section \ref{s2} we briefly review representations for spectral zeta functions as derived in \cite{GK18} and we apply them to the zeta function of Riemann. New integral representations for the Riemann zeta function are found and the well-known properties, namely values at even negative and positive integers are easily reproduced. In addition, we derive new representations for the value of the Riemann zeta function at positive odd integers. Typical examples we derive are
\begin{align}
\begin{split} 
\zeta(z) = \sin(\pi z/2) \pi^{z-1} \int_0^{\infty} ds \, s^{-z} [\coth(s) - \coth_n(s)],& \\
\Re(z) \in (\max(1,2n), 2n+2), \; n \in \bbN_0,&      \lb{1.5}
\end{split} 
\end{align}
where
\begin{align}
\begin{split} 
\coth_0(z) &= \f{1}{z}, \quad z \in \bbC \backslash \{0\},   \\
\coth_n(z) &= \f{1}{z} +\sum_{k=1}^n \frac{2^{2k} B_{2k}} {(2k)!} z^{2k-1}, \quad 
z \in \bbC \backslash \{0\}, \; n \in \bbN, \lb{1.6}
\end{split} 
\end{align} 
with $B_{m}$ the Bernoulli numbers (cf.\ \eqref{A.30a}--\eqref{A.33}), implying  
\begin{align} 
\zeta(3) &= -\pi^2 \, \int_0^\infty \, ds \, s^{-3} [\coth (s) - (1/s) - (s/3)],    \\[3mm] 
\zeta(5) &= \pi^4 \, \int_0^\infty \, ds \, s^{-5} \big[\coth (s) - (1/s) - (s/3) + \big(s^3/45\big)\big],    
\\[3mm] 
\zeta(7) & = -\pi^6 \, \int_0^\infty \, ds \, s^{-7} \big[\coth (s) - (1/s) - (s/3) 
+\big(s^3/45\big) - \big(2s^5/945\big)\big],   \\
& \text{etc.}    \no
\end{align} 

Finally, Appendix \ref{sA} summarizes known results about the Riemann zeta function putting 
the results we found in some perspective.

\section{Computing Traces and  the Riemann $\zeta$-Function} \lb{s2}

After a brief discussion of spectral zeta functions associated with self-adjoint operators with purely discrete spectra, 
we turn to applications of spectral trace formulas to the Riemann zeta function. 

We start by following the recent paper \cite{GK18} and briefly discuss spectral $\zeta$-functions of self-adjoint 
operators $S$ with a trace class resolvent (and hence a purely discrete spectrum).

Below we will employ the following notational conventions: A separable, complex Hilbert space is denoted 
by $\cH$, $I_{\cH}$ represents the identity operator in $\cH$; the resolvent set and spectrum of a closed operator 
$T$ in $\cH$ are abbreviated by $\rho(T)$ and $\sigma(T)$, respectively; the Banach space of trace class operators 
on $\cH$ is denoted by $\cB_1(\cH)$, and the trace of a trace class operator $A \in \cB_1(\cH)$ is abbreviated by 
$\tr_{\cH}(A)$. 

\begin{hypothesis} \lb{h2.1}
Suppose $S$ is a self-adjoint operator in $\cH$, bounded from below, satisfying
\begin{equation}
(S- z I_{\cH})^{-1} \in \cB_1(\cH)
\end{equation}
for some $($and hence for all\,$)$ $z \in \rho(S)$.
We denote the spectrum of $S$ by $\sigma(S) = \{\lambda_j\}_{j \in J}$ $($with $J \subset \bbZ$ an appropriate index set\,$)$, with every eigenvalue repeated according to its multiplicity.
\end{hypothesis}

Given Hypothesis \ref{h2.1}, the spectral
zeta function of $S$ is then defined by
\begin{equation}
\zeta (z; S) = \sum_{\substack{j \in J \\ \lambda_j \neq 0}} \lambda_j^{-z}    \lb{2.2}
\end{equation}
for $\Re(z) > 0$ sufficiently large such that \eqref{2.2} converges absolutely.

Next, let $P(0; S)$ be the spectral projection of $S$ corresponding to
the eigenvalue $0$ and denote by $m(\lambda_0; S)$ the multiplicity of the eigenvalue $\lambda_0$ of $S$, in particular,
\begin{equation}
m(0; S) = \dim(\ker(S)).
\end{equation}
In addition, we introduce the simple contour $\gamma$ encircling $\sigma(S) \backslash \{0\}$  in a
counterclockwise manner so as to dip under (and hence avoid) the point $0$ (cf.\ Figure \ref{fig1}). In fact, following \cite{KM04} (see also \cite{KM03}), we will henceforth choose as the branch cut of $w^{-z}$ the ray
\begin{equation}
R_{\theta} = \big\{w = t e^{i \theta} \big| t \in [0,\infty)\big\} \quad
\theta \in (\pi/2, \pi),  \lb{2.4}
\end{equation}
and note that the contour $\gamma$ avoids any contact with $R_{\theta}$
(cf.\ Figure \ref{fig1}). 
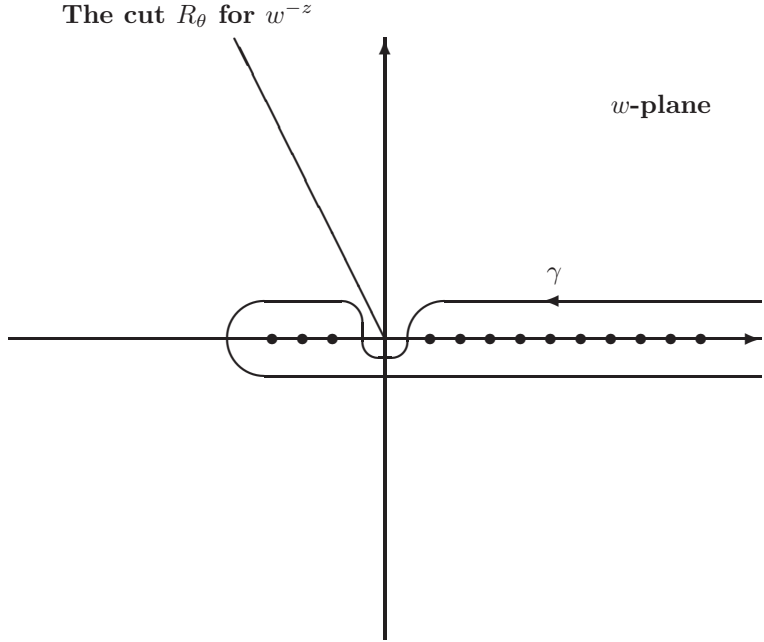
\begin{figure}[H]
\setlength{\unitlength}{1cm}

\begin{center}

\begin{picture}(20,10)(0,0)

\put(0,0){\setlength{\unitlength}{1.0cm}
\begin{picture}(10,7.5)
\thicklines

\put(0,4){\vector(1,0){10}} \put(5.0,0){\vector(0,1){8}}
\put(10.0,4){\oval(9.4,1)[tl]} \put(5,4){\oval(0.6,0.5)[b]}
\put(4.4,4){\oval(0.6,1)[tr]} \put(4.4,4){\oval(3.0,1)[l]}
\put(7.5,4.5){\vector(-1,0){.4}}\put(4.4,3.5){\line(1,0){5.6}}
\put(5.0,4.0){\line(-1,2){2.0}} \put(.7,8.2){{\bf The cut $R_{\theta}$ for
$w^{-z}$}} \multiput(5.6,4)(.4,0){10}{\circle*{.15}}
\multiput(3.5,4)(.4,0){3}{\circle*{.15}}
\put(8.0,7){{\bf $w$-plane}}
\put(7.15,4.8){{\bf $\gamma$}}

\end{picture}}

\end{picture}

\caption{Contour $\gamma$ in the complex $w$-plane.} \lb{fig1}

\end{center}

\end{figure}

\begin{lemma} \lb{l2.2}
In addition to Hypothesis \ref{h2.1} and the counterclockwise oriented contour $\gamma$ just described $($cf.\ Figure \ref{fig1}$)$, suppose that
$\big|{\tr}_{\cH}\big((S- z I_{\cH})^{-1} [I_{\cH} - P(0; S)]\big)\big|$
is polynomially bounded with respect to $z$ on $\gamma$. Then
\begin{equation}
\zeta(z; S) = - (2 \pi i)^{-1}  \ointctrclockwise_{\gamma} dw \, w^{-z}
\big[\tr_{\cH}\big((S - w I_{\cH})^{-1}\big) + w^{-1} m(0; S)\big]   \lb{2.6}
\end{equation}
for $\Re(z) > 0$ sufficiently large.
\end{lemma}

We note in passing that one could also use a semigroup approach via
\begin{align}
\begin{split}
\zeta(z; S) &= \Gamma(z)^{-1} \int_0^{\infty} dt \, t^{z - 1} \tr_{\cH}\big(e^{- t S}
[I_{\cH} - P(0; S)]\big)   \lb{2.12} \\
&= \Gamma(z)^{-1} \int_0^{\infty} dt \, t^{z - 1} \big[\tr_{\cH}\big(e^{- t S}\big) - m(0; S)\big],
\end{split}
\end{align}
for $\Re(z) > 0$ sufficiently large.

It is natural to continue the computation leading to \eqref{2.6}
and now deform the contour $\gamma$ so as to ``hug'' the branch cut $R_{\theta}$, but this requires the right asymptotic behavior of $\tr_{\cH}\big((S - w I_{\cH})^{-1} [I_{\cH} - P(0; S)]\big)$ as
$|w| \to \infty$ as well as $|w| \to 0$. This applies, in particular, to cases where $S$ is strictly positive and one thus chooses the branch cut along the negative axis, that is, employs the cut $R_{\pi}$, where
\begin{equation}
R_{\pi} = (- \infty, 0],  \lb{2.7}
\end{equation}
and choosing the contour $\gamma$ to encircle $R_{\pi}$ clockwise. This renders \eqref{2.6} into the 
following expression 
\begin{align}
\zeta(z;S) &= \f{\sin(\pi z)}{\pi} \int_0^{\infty} dt \, t^{-z} \tr_{\cH}\big((S  + t I_{\cH})^{-1}\big)   \no \\
&= \tr_{\cH}\bigg(\f{\sin(\pi z)}{\pi} \int_0^{\infty} dt \, t^{-z} (S  + t I_{\cH})^{-1}\bigg)   \no \\
&= \tr_{\cH} \big(S^{-z}\big),    \lb{2.8}
\end{align}
employing the fact,
\begin{equation}
S^{-z} = \f{\sin(\pi z)}{\pi} \int_0^{\infty} dt \, t^{-z} (S  + t I_{\cH})^{-1}, \quad \Re(z) \in (0,1),    \lb{2.9}
\end{equation}
whenever $S \geq 0$ in $\cH$, with $\ker(S) = \{0\}$ (see, e.g., \cite[Proposition~3.2.1\,d)]{Ha06}). \\[1mm]
\noindent
{\bf Note.} While \eqref{2.9} is rigorous, the manipulations in \eqref{2.8} are formal and subject to appropriate convergence and trace class hypotheses which will affect the possible range of $\Re(z)$. \hfill $\diamond$

These hypotheses are easily shown to be satisfied when discussing the one-dimensional Dirichlet Laplacian
$- \Delta_D$ in $L^2((0,1); dx)$, 
\begin{align}
& - \Delta_D = - \f{d^2}{dx^2},   \no \\
& \, \dom(- \Delta_D) = \big\{u \in L^2((0,1); dx) \, \big| \, u, u' \in AC_{loc}([0,1]); \, u(0) = 0 = u(1);    \lb{2.9a} \\
& \hspace*{8.2cm} u'' \in L^2((0,1); dx)\big\}   \no
\end{align}
(here $AC([0,1])$ denotes the set of absolutely continuous functions on $[0,1]$).

Recalling Riemann's celebrated zeta function (see Appendix \ref{sA} for more details),
\begin{equation}
\zeta(z) = \sum_{k \in \bbN} k^{-z}, \quad z \in \bbC, \, \Re(z) > 1.     \lb{2.9b}\\
\end{equation}
we start with the following result. 

\begin{lemma} \lb{l2.3}
Let $\Re(z) \in (1,2)$, then
\begin{equation}
\zeta(z) = \sin(\pi z/2) \pi^{z-1} \int_0^{\infty} ds \, s^{-z - 1} [s \coth(s) - 1].        \lb{2.9c}
\end{equation}
\end{lemma}
\begin{proof}
Since the eigenvalue problem for $- \Delta_D$ reads
\begin{equation}
- \Delta_D u_k = \lambda_k u_k, \quad  k \in \bbN,
\end{equation}
with
\begin{equation}
u_k(x) = 2^{1/2} \sin(k \pi x), \; \|u_k\|_{L^2((0,1);dx)} = 1, \quad \lambda_k = (\pi k)^2, \quad k \in \bbN,
\end{equation}
and all eigenvalues $\lambda_k$, $k \in \bbN$, are simple, \eqref{2.8} works as long as $\Re(z) \in ((1/2),1)$ 
and one obtains (see also \cite[p.~94]{Di61})
\begin{align}
& \zeta(z; - \Delta_D) = \tr_{L^2((0,1); dx)} \big((- \Delta_D)^{-z}\big) = \sum_{k \in \bbN} (\pi k)^{- 2 z}
= \pi^{-2z} \zeta(2 z)   \no \\
&\quad = \f{\sin(\pi z)}{\pi} \int_0^{\infty} dt \, t^{-z}
\tr_{L^2((0,1); dx)}\big((- \Delta_D + t I_{L^2((0,1); dx)})^{-1}\big)   \no \\
&\quad = \f{\sin(\pi z)}{\pi} \int_0^{\infty} dt \, t^{-z} \int_0^1 dx \, t^{-1/2} \big[\sinh(t^{1/2}\big)\big]^{-1}
\sinh\big(t^{1/2}x\big) \sinh\big(t^{1/2}(1-x)\big)    \no \\
&\quad =  \f{\sin(\pi z)}{2 \pi} \int_0^{\infty} dt \, t^{- z - 1} \big[t^{1/2} \coth \big(t^{1/2}\big) - 1\big],
\quad  \Re(z) \in ((1/2),1).   \lb{2.10}
\end{align}
Here we used
\begin{align}
\begin{split}
& (- \Delta_D - z I_{L^2((0,1); dx)})^{-1}(z,x,x')    \\
& \quad = \f{1}{z^{1/2}\sin(z^{1/2})} \begin{cases}
\sin(z^{1/2}x) \sin(z^{1/2}(1-x')), & 0 \leq x \leq x' \leq 1, \\
\sin(z^{1/2}x') \sin(z^{1/2}(1-x)), & 0 \leq x' \leq x \leq 1,
\end{cases}   \\
& \hspace*{7.3cm} z \in \bbC \big \backslash \big\{\pi^2 k^2\big\}_{k \in \bbN},
\end{split}
\end{align}
and \cite[2.4254]{GR80}
\begin{align}
\begin{split}
& \int^x dt \, \sinh(at + b) \sinh(at + c)   \\
& \quad = - (t/2) \cosh(b - c) + (4a)^{-1} \sinh(2 a t + b + c) + C.
\end{split}
\end{align}
Since
\begin{equation}
\big[t^{1/2} \coth\big(t^{1/2}\big) - 1\big] = \begin{cases} \Oh(t), & t \downarrow 0, \\
\Oh\big(t^{1/2}\big), & t \uparrow \infty,
\end{cases}
\end{equation}
\eqref{2.10} is well-defined for $ \Re(z) \in ((1/2),1)$. Thus, the elementary change of variables $t = s^2$ yields
\eqref{2.9c}.
\end{proof}

\begin{remark} \lb{r2.4}
Representation (\ref{2.9c}) is suitable to observe the well-known properties 
\begin{equation} 
\zeta (0) = - 1/2, \quad \zeta (-2n) =0, \; n\in \bbN.   \lb{2.13}
\end{equation}
To this end, one notes that the restriction $\Re (z) >1$ results from the $s \to \infty$ behavior of the 
integrand in \eqref{2.9c}. Explicitly, one has
\begin{equation} 
s \coth (s )-1 \underset{s \to \infty}{=} s-1 + \Oh\big( e^{-2s} \big), 
\end{equation} 
from which one infers that 
\begin{align}
\zeta (z) &= \sin (\pi z /2) \pi^{z-1} \int\limits_1^\infty ds \, s^{-z-1} (s-1) + E(z) \nonumber\\
&= \sin (\pi z/2) \pi^{z-1} \left[ \frac 1 {z-1} - \frac 1 z \right] + E(z) , 
\end{align}
where $E(\, \cdot \,)$ is entire and 
\begin{equation}
E(- 2 n) =0, \quad n\in \bbN_0. 
\end{equation} 
This immediately implies \eqref{2.13}.

The values at positive even integers, $\zeta (2m)$ for $m\in\bbN$, are best obtained using representation
(\ref{2.6}); see Remark \ref{r2.6}. \hfill $\diamond$
\end{remark}

One can generalize (\ref{2.9}) as follows:
\begin{align}
\begin{split}
S^{-z} = \f{\Gamma(m)}{\Gamma(n - z)\Gamma(m - n + z)} S^{m - n} \int_0^{\infty} dt \,
t^{n - 1 - z} (S + t I_{\cH})^{-m},&    \\
\Re(z) \in (n - m, n), \; m, n \in \bbN.&     \lb{2.14}
\end{split}
\end{align}
Formally, this now yields
\begin{align}
\zeta(z;S) &= \f{\Gamma(m)}{\Gamma(n -z)\Gamma(m - n + z)} \int_0^{\infty} dt \, t^{n - 1 -z}
\tr_{\cH}\big(S^{m - n}(S  + t I_{\cH})^{-m}\big)   \no \\
&= \tr_{\cH}\bigg(\f{\Gamma(m)}{\Gamma(n -z)\Gamma(m - n + z)} S^{m - n}
\int_0^{\infty} dt \, t^{n - 1 -z} (S  + t I_{\cH})^{-m}\bigg)   \no \\
&= \tr_{\cH} \big(S^{-z}\big).   \lb{2.15}
\end{align}
The case $m = n$ appears to be the simplest and yields
\begin{align}
\zeta(z;S) &= \f{\Gamma(n)}{\Gamma(n -z)\Gamma(z)} \int_0^{\infty} dt \, t^{n - 1 -z}
\tr_{\cH}\big((S  + t I_{\cH})^{-n}\big)   \no \\
&= \tr_{\cH}\bigg(\f{\Gamma(n)}{\Gamma(n -z)\Gamma(z)}
\int_0^{\infty} dt \, t^{n - 1 -z} (S  + t I_{\cH})^{-n}\bigg)   \no \\
&= \tr_{\cH} \big(S^{-z}\big).   \lb{2.16}
\end{align}
{\bf Note.} Again, \eqref{2.14} is rigorous, but \eqref{2.15}, \eqref{2.16} are subject to
``appropriate'' convergence and trace class hypotheses. \hfill $\diamond$

For $- \Delta_D$, \eqref{2.16} indeed works for $n \in \bbN$ as long as $\Re(z) \in ((1/2),n)$ and one obtains
\begin{align}
\begin{split}
& \zeta(z; - \Delta_D) = \tr_{L^2((0,1); dx)} \big((- \Delta_D)^{-z}\big) = \sum_{k \in \bbN} (\pi k)^{- 2 z}
= \pi^{-2z} \zeta(2 z)    \\
&\quad = \f{\Gamma(n)}{\Gamma(n - z) \Gamma(z)} \int_0^{\infty} dt \, t^{n - 1- z}
\tr_{L^2((0,1); dx)}\big((- \Delta_D + t I_{L^2((0,1); dx)})^{-n}\big).    \lb{2.17}
\end{split}
\end{align}
For $n = 2$ this includes $z = 3/2$ and hence leads to a formula for $\zeta(3)$. However, we prefer an
alternative approach based on \cite[Theorem~3.4\,$(i)$]{GK18} that applies to $- \Delta_D$ and yields
the following results:

\begin{lemma} \lb{l2.5}
Let $\Re(z) \in (1,4)$, then
\begin{equation}
\zeta(z) = \f{\sin(\pi z/2)}{(2 - z)} \pi^{z-1} \int_0^{\infty} ds \, s^{-z - 1}
\Big[s \coth(s) + s^2 [\sinh(s)]^{-2} - 2\Big].   \lb{2.24} \\
\end{equation}
In particular,
\begin{align}
\zeta(2) &= \f{\pi^2}{2} \int_0^{\infty} ds \, s^{- 3}
\Big[s \coth(s) + s^2 [\sinh(s)]^{-2} - 2\Big] = .... = \pi^2/6,      \lb{2.25} \\
\zeta(3) &= \pi^{2} \int_0^{\infty} ds \, s^{- 4}
\Big[s \coth(s) + s^2 [\sinh(s)]^{-2} - 2\Big].     \lb{2.26}
\end{align}
\end{lemma}
\begin{proof}
Employing \cite[Theorem~3.4\,$(i)$]{GK18},
\begin{align}
\begin{split}
\tr_{L^2((0,1); dx)} \big((- \Delta_D - z I_{L^2((0,1); dx)})^{-1}\big)
= - (d/dz) \ln\big(z^{-1/2} \sin\big(z^{1/2}\big)\big),&   \lb{2.18} \\
z \in \bbC \backslash \big\{\pi^2 k^2\big\}_{k \in \bbN}&
\end{split}
\end{align}
(see also \eqref{A.30}), one confirms that 
\begin{equation}
\zeta(2)/\pi^2 = \lim_{z \to 0} \sum_{k \in \bbN} \big(\pi^2 k^2 - z\big)^{-1} = 1/6.
\end{equation}
Actually, setting $z = - t$ in \eqref{2.18} yields
\begin{align}
& \tr_{L^2((0,1); dx)} \big((- \Delta_D + t I_{L^2((0,1); dx)})^{-1}\big)
= (d/dt) \ln\big(t^{-1/2} \sinh\big(t^{1/2}\big)\big)   \no \\
&\quad = (2 t)^{-1} \big[t^{1/2} \coth\big(t^{1/2}\big) - 1\big], \quad t > 0,
\end{align}
and hence confirms \eqref{2.10}. Continuing that process, one notes that
\begin{align}
& \tr_{L^2((0,1); dx)} \big((- \Delta_D + t I_{L^2((0,1); dx)})^{-2}\big)
= - \big(d^2/dt^2\big) \ln\big(t^{-1/2} \sinh\big(t^{1/2}\big)\big)   \no \\
&\quad = \big(4 t^2\big)^{-1} \Big[t^{1/2} \coth\big(t^{1/2}\big)
+ t \big[\sinh\big(t^{1/2}\big)\big]^{-2}- 2\Big], \quad t > 0.    \lb{2.21}
\end{align}
Insertion of \eqref{2.21} into \eqref{2.17} taking $n = 2$ then yields
\begin{align}
& \zeta(z; - \Delta_D) = \tr_{L^2((0,1); dx)} \big((- \Delta_D)^{-z}\big) = \sum_{k \in \bbN} (\pi k)^{- 2 z}
= \pi^{-2z} \zeta(2 z)    \no \\
&\quad = \f{1}{\Gamma(2 - z) \Gamma(z)} \int_0^{\infty} dt \, t^{1- z}
\tr_{L^2((0,1); dx)}\big((- \Delta_D + t I_{L^2((0,1); dx)})^{-2}\big)     \no \\
&\quad = \f{\sin(\pi z)}{4\pi (1 - z)} \int_0^{\infty} dt \, t^{- 1 - z} \Big[t^{1/2} \coth\big(t^{1/2}\big)
+ t \big[\sinh\big(t^{1/2}\big)\big]^{-2}- 2\Big],    \lb{2.22} \\
& \hspace*{7.62cm}  \Re(z) \in ((1/2),2).   \no
\end{align}
Since
\begin{equation}
\Big[t^{1/2} \coth\big(t^{1/2}\big)
+ t \big[\sinh\big(t^{1/2}\big)\big]^{-2}- 2\Big] = \begin{cases} \Oh\big(t^2\big), & t \downarrow 0, \\
\Oh\big(t^{1/2}\big), & t \uparrow \infty,
\end{cases}
\end{equation}
\eqref{2.22} is well-defined for $ \Re(z) \in ((1/2),2)$. Thus, the elementary change of variables $t = s^2$ yields
\eqref{2.24}--\eqref{2.26}.
\end{proof}

\begin{remark} \lb{r2.6}
Employing (\ref{2.18}) in (\ref{2.6}), one finds the representation 
\begin{equation}
\zeta (z; - \Delta_D ) = (2\pi i)^{-1} \ointctrclockwise_\gamma dw \, w^{-z} 
\frac{ 1 - w^{1/2} \cot \big(w^{1/2}\big)} {2 w },
\end{equation}
where the counterclockwise contour $\gamma$ can be chosen to consist of a circle $\gamma_\epsilon$ of radius $\epsilon < \pi$ 
and straight lines $\gamma_1$, respectively $\gamma_2$, just above, respectively just below, the negative $x$-axis. For $z=m$, $m\in \bbN$, 
contributions from $\gamma_1$ and $\gamma_2$ cancel each other and thus
\begin{equation}
\zeta (m; - \Delta_D ) =  (2\pi i)^{-1} \ointctrclockwise_{\gamma_\epsilon} dw \, w^{-m} 
\frac{ 1- w^{1/2} \cot \big(w^{1/2}\big)} {2 w }.    
\end{equation}
This integral is easily computed using the residue theorem. From the Taylor series \cite{GR80}
\begin{equation}
\frac{ 1-  w^{1/2} \cot \big(w^{1/2}\big)} {2 w} = \sum_{k=1}^\infty \frac{ 2^{2k-1} |B_{2k}|} {(2k)!} w^{k-1}, 
\quad w \in \bbC, \; 0 < |w| < \pi^2
\end{equation}
(with $B_{m}$ the Bernoulli numbers, cf.\ \eqref{A.30a}--\eqref{A.33}), the relevant term is $k=m$ and 
\begin{equation}
\zeta (m; - \Delta_D ) = \frac{2^{2m-1} |B_{2m}|} {(2m)!}, 
\end{equation}
implying Euler's celebrated result, 
\begin{equation}
\zeta (2m) = \frac{2^{2m-1} \pi^{2m} |B_{2m}|} {(2m)!}, \quad m \in \bbN. 
\end{equation}

This procedure works in a much more general context and allows for the computation of traces of powers of Sturm--Liouville operators in a fairly straightforward fashion; this will be revisited elsewhere.  \hfill $\diamond$ 
\end{remark}

\begin{remark} \lb{r2.7}
An elementary integration by parts of the term $s^{-2} \coth(s)$ in \eqref{2.25} indeed verifies once more that $\zeta(2) = \pi^2/6$. The same integration by parts in \eqref{2.26} fails to render the integral trivial (as it obviously should not be trivial). Indeed,
\begin{align}
& \int_{\varepsilon}^R ds \, \big\{s^{-2} \coth(s) + s^{-1} [\sinh(s)]^{-2} - 2 s^{-3} \big\}    \no \\
& \quad = \int_{\varepsilon}^R ds \,
\big\{ - \big[(d/ds) s^{-1}\big] \coth(s) + s^{-1} [\sinh(s)]^{-2} - 2 s^{-3} \big\}    \no \\
& \quad = - s^{-1} \coth(s)\Big|_{\varepsilon}^R + \int_{\varepsilon}^R ds \, (-2) s^{-3}
\underset{\varepsilon \downarrow 0, \, R \uparrow \infty}{\longrightarrow} = \f{1}{3}.
\end{align}
Applying the same strategy to \eqref{2.26} yields
\begin{align}
& \int_{\varepsilon}^R ds \, \Big[s^{-3} \coth(s) + s^{-2} [\sinh(s)]^{-2} - 2 s^{-4} \Big]    \no \\
& \quad = \int_{\varepsilon}^R ds \,
\Big[ - (1/2) \big[(d/ds) s^{-2}\big] \coth(s) + s^{-2} [\sinh(s)]^{-2} - 2 s^{-4} \Big]    \no \\
& \quad = - (1/2) s^{-2} \coth(s)\Big|_{\varepsilon}^R
+ \int_{\varepsilon}^R ds \, \Big[(1/2) s^{-2} [\sinh(s)]^{-2} -2 s^{-4}\Big],
\end{align}
and hence the expected nontrivial integral. We note once more that
\begin{equation}
\Big[s \coth(s) + s^2 [\sinh(s)]^{-2} - 2\Big] \underset{s \downarrow 0}{=} \Oh\big(s^4\big),
\end{equation}
rendering \eqref{2.26} well-defined. \hfill $\diamond$
\end{remark}

The following alternative (though, equivalent) approach to $\zeta(z)$ is perhaps a bit more streamlined:

\begin{theorem} \lb{t2.8}
Let $n \in \bbN_0$, $0 < \Re(z) < 1$, and\footnote{The condition $\Re(2n + 2z) > 1$ takes effect only for 
$n=0$, that is, we assume $(1/2) < \Re(z) < 1$ if $n=0$.} $\Re(2n + 2z) > 1$.  
Then\footnote{The second formula is mentioned since it appears
to be advantageous (cf.\ \eqref{2.24}--\eqref{2.26})
to substitute $t = s^2$ after one performed the $n$ differentiations w.r.t. $t$ in the 1st line of \eqref{2.31}.}
\begin{align}
\zeta(2n + 2z) &= \f{(-1)^n \pi^{2(n+z)}}{\Gamma(1-z)\Gamma(n+z)}
\int_0^{\infty} dt \,  \, t^{-z} \f{d^n}{dt^n} \Big[(2 t)^{-1} \big[t^{1/2} \coth\big(t^{1/2}\big) - 1\big]\Big]   \no \\
&= \f{(-1)^n 2^{-n} \pi^{2(n+z)}}{\Gamma(1-z)\Gamma(n+z)}
\int_0^{\infty} ds \,  \, s^{1 - 2z} \bigg(\f{d}{s ds}\bigg)^n \Big[s^{-2} [s \coth(s) - 1]\Big].     \lb{2.31}
\end{align}
In addition to $\zeta(3)$ in \eqref{2.26} one thus obtains similarly,
\begin{align}
\zeta(5) &= \f{\pi^4}{3} \int_0^{\infty} ds \, s^{-6} \Big[2 s^3 \coth(s) [\sinh(s)]^{-2} + 3 s^2 [\sinh(s)]^{-2}
\no \\
& \hspace*{2.7cm} + 3 s \coth(s) - 8\Big],    \\
\zeta(7) &= \f{\pi^6}{15} \int_0^{\infty} ds \, s^{-8} \Big[4 s^4 [\coth(s)]^2 [\sinh(s)]^{-2}    \no \\
& \hspace*{2.7cm}
+ 2 s^4 [\sinh(s)]^{-4} + 12 s^3 \coth(s) [\sinh(s)]^{-2}     \no \\
& \hspace*{2.7cm}
+ 15 s^2 [\sinh(s)]^{-2} + 15 s \coth(s) -48\Big],     \\
& \text{etc.}    \no
\end{align}
\end{theorem}
\begin{proof} Assume that $n \in \bbN_0$, $0 < \Re(z) < 1$, and $\Re(2n + 2z) > 1$. Then, 
\begin{align}
& \int_0^{\infty} dt \, t^{-z} \f{d^n}{dt^n} \Big[\tr_{L^2((0,1); dx)} \big((- \Delta_D + t I_{L^2((0,1); dx)})^{-1}\big)\Big]
\no \\
& \quad = \int_0^{\infty} dt \, t^{-z} \f{d^n}{dt^n} \Big[(2 t)^{-1} \big[t^{1/2} \coth\big(t^{1/2}\big) - 1\big]\Big]
\no \\
& \quad = \sum_{k \in \bbN} \int_0^{\infty} dt \, t^{-z} \frac{d^n}{dt^n} \Big[\big(\pi^2 k^2 + t\big)^{-1}\Big]  \no \\
& \quad = \sum_{k \in \bbN} \int_0^{\infty} dt \, \f{t^{-z} (-1)^n n!}{\big(\pi^2 k^2 + t\big)^{n + 1}}  \no \\
& \quad = (-1)^n n! \sum_{k \in \bbN} \big(\pi^2 k^2\big)^{-z - n} \int_0^{\infty} du \, \f{u^{-z}}{(1 + u)^{n+1}} \no \\
& \quad = (-1)^n n! \pi^{-2 (n + z)} \zeta(2(n + z)) \f{\Gamma(1-z)\Gamma(n+z)}{\Gamma(n+1)}   \no \\
& \quad = (-1)^n \pi^{-2 (n + z)} \zeta(2(n + z)) \Gamma(1-z)\Gamma(n+z),      \lb{2.47} 
\end{align}
resulting in \eqref{2.31}. (The condition $(1/2) < \Re(z) < 1$ if $n=0$ guarantees convergence of the sum over $k$ in \eqref{2.47}.) 
\end{proof}

Alternatively, one can attempt to analytically continue the equation
\begin{equation}
\zeta (z) = \sin(\pi z /2) \pi^{z-1} \int_0^\infty ds \, s^{-z-1} [ s \coth (s) -1], \quad \Re(z) \in (1,2),  \lb{2.48}
\end{equation}
to the region $\Re(z) \geq 2$. For this purpose we first introduce 
\begin{align}
\coth (z) &= \f{1}{z} +\sum_{k=1}^\infty \frac{2^{2k} B_{2k}} {(2k)!} z^{2k-1}, \quad z \in \bbC, \; 
0 < |z| < \pi,   \lb{2.49} \\
\begin{split} 
\coth_0(z) &= \f{1}{z}, \quad z \in \bbC \backslash \{0\},   \\
\coth_n(z) &= \f{1}{z} +\sum_{k=1}^n \frac{2^{2k} B_{2k}} {(2k)!} z^{2k-1}, \quad 
z \in \bbC \backslash \{0\}, \; n \in \bbN. \lb{2.50}
\end{split} 
\end{align}

\begin{theorem} \lb{t2.9} 
Let $n \in \bbN_0$, then,
\begin{equation}
\zeta(z) = \sin(\pi z/2) \pi^{z-1} \int_0^{\infty} ds \, s^{-z} [\coth(s) - \coth_n(s)], \quad 
\Re(z) \in (\max(1,2n), 2n+2).      \lb{2.51}
\end{equation}
In particular, 
\begin{align} 
\begin{split}
\zeta(3) &= -\pi^2 \, \int_0^\infty \, ds \, s^{-3} [\coth (s) - \coth_1 (s)] \\[1mm]  
&= -\pi^2 \, \int_0^\infty \, ds \, s^{-3} [\coth (s) - (1/s) - (s/3)],   
\end{split} \\[3mm] 
\begin{split} 
\zeta(5) &= \pi^4 \, \int_0^\infty \, ds \, s^{-5} [\coth (s) - \coth_2 (s)] \\[1mm]  
&= \pi^4 \, \int_0^\infty \, ds \, s^{-5} \big[\coth (s) - (1/s) - (s/3) + \big(s^3/45\big)\big],    
\end{split} \\[3mm] 
\begin{split} 
\zeta(7) &= -\pi^6 \, \int_0^\infty \, ds \, s^{-7} [\coth (s) - \coth_3 (s)] \\[1mm]  
& = -\pi^6 \, \int_0^\infty \, ds \, s^{-7} \big[\coth (s) - (1/s) - (s/3) + \big(s^3/45\big) - \big(2s^5/945\big)\big],  
\end{split} \\
& \text{etc.}    \no
\end{align} 
\end{theorem}
\begin{proof}

When trying to analytically continue \eqref{2.48} to the right, one notices that it is the
small $s$-behavior of the integrand that
invalidates this representation. We therefore split the integral at some point $a > 0$ and write
\begin{align}
\begin{split}
\zeta (z) &= \sin(\pi z/2) \pi^{z-1} \int_a^\infty ds \, s^{-z-1} [ s \coth (s) -1]   \\
& \quad +\sin(\pi z/2) \pi^{z-1} \int_0^a ds \, s^{-z-1} [ s \coth (s) -1],
\end{split}
\end{align}
where the first integral is well-defined for $\Re (z) >1$, and the second for $\Re (z) < 2$. In order to analytically continue the second integral to the right, one writes
\begin{align}
& \int_0^a ds \, s^{-z-1} [ s \coth (s) -1] = \int_0^a ds \, s^{-z-1} \bigg[ s \coth (s) - 1-\sum_{k=1}^n \frac{2^{2k} B_{2k}}{(2k)!} s^{2k} \bigg]    \no \\
& \qquad + \int_0^a ds \, s^{-z-1} \sum_{k=1}^n \frac{ 2^{2k} B_{2k}}{(2k)!} s^{2k}  \no \\
& \quad = \int_0^a ds \, s^{-z-1} \bigg[ s \coth (s) - 1-\sum_{k=1}^n \frac{2^{2k} B_{2k}}{(2k)!} s^{2k} \bigg]  
 + \sum_{k=1}^n \frac{ 2^{2k} B_{2k}}{(2k)!} \int_0^a ds \, s^{-z-1+2k}    \no \\
& \quad = \int_0^a ds \, s^{-z-1} \bigg[ s \coth (s) - 1-\sum_{k=1}^N \frac{2^{2k} B_{2k}}{(2k)!} s^{2k} \bigg]  
 + \sum_{k=1}^n \frac{ 2^{2k} B_{2k}}{(2k)!} \frac{ a^{2k-z}}{2k-z},
\end{align}
valid for $1 < \Re (z) < 2n+2$, $z \notin \{2\ell\}_{1 \leq \ell \leq n}$.

In summary, up to this point we have shown that for $1 < \Re (z) < 2n+2$, 
$z \notin \{2\ell\}_{1 \leq \ell \leq n}$, and for $a > 0$,  one has
\begin{align}
\zeta (z) &= \sin(\pi z/ 2) \pi^{z-1} \bigg\{ \int_a^\infty ds \, s^{-z-1} [ s \coth (s) -1]
\no \\
& \quad + \int_0^a ds \, s^{-z-1} \bigg[ s \coth (s) - 1 - \sum_{k=1}^n \frac{ 2^{2k} B_{2k}}{(2k)!} s^{2k} \bigg]
\no \\
&\quad + \sum_{k=1}^n \frac{ 2^{2k} B_{2k} }{(2k)!} \frac{a^{2k-z}}{2k-z} \bigg\}.   \lb{2.57} 
\end{align}
Restricting $z$ to $\Re(z) \in (\max(1,2n), 2n+2)$ and performing the limit $a \to \infty$ in \eqref{2.57}, observing that the first and third terms on the right-hand 
side of \eqref{2.57} vanish in the limit, proves \eqref{2.51}. 
\end{proof}

One notes that for $z=2n$ the first two lines in \eqref{2.57} as well as all terms $k\neq n$ vanish and one 
confirms Euler's celebrated formula
\begin{equation}
\zeta (2n) = \lim_{z\to 2n} \bigg[\sin(\pi z/ 2) \pi^{z-1} \frac{2^{2n} B_{2n}}{(2n)!} \frac{a^{2n-z}}{2n-z}\bigg] = \frac{ 2^{2n-1} |B_{2n}| \pi^{2n}}{(2n)!}, \quad n \in \bbN.
\end{equation}

Finally, one can take these investigations one step further as follows. Introducing 
\begin{align}
F(z) &= \ln \big(z^{-1/2} \sinh\big(z^{1/2}\big)\big) 
= \sum_{k=1}^\infty \, \dfrac{2^{2k} B_{2k}}{2k (2k)!} \, z^k, \quad z \in \bbC, \; |z|<\pi,   \\
F_n(z) &= \sum_{k=1}^n \, \dfrac{2^{2k} B_{2k}}{2k (2k)!} \, z^k, \quad z \in \bbC, \; n \in \bbN,   
\end{align}
one can show the following result.

\begin{theorem} \lb{t2.10} 
Let $n \in \bbN_0$, then, 
\begin{align} 
\begin{split} 
\zeta(z) = (z/2) \pi^{z-1} \sin (\pi z/2) \, \int_0^\infty \, dt \, t^{-z/2-1} \, [F(t) - F_n(t)],& \\ 
\Re(z) \in (\max(2n,1),2n+2).& 
\end{split} 
\end{align} 
In particular,
\begin{align} 
\begin{split} 
\zeta(3) &= -3 \pi^2 \, \int_0^\infty \, ds \, s^{-4}  \big[F\big(s^2\big) -F_1\big(s^2\big)\big] \\[1mm]  
&= -3 \pi^2 \, \int_0^\infty \, ds \, s^{-4} \big[\ln \big(s^{-1}\sinh(s)\big) - \big(s^2/6\big)\big], 
\end{split} \\[3mm] 
\begin{split} 
\zeta(5) &= 5 \pi^4 \, \int_0^\infty \, ds \, s^{-6} \big[F\big(s^2\big) - F_2\big(s^2\big)\big] \\[1mm] 
&= 5 \pi^4 \, \int_0^\infty \, ds \, s^{-6} \big[\ln\big(s^{-1}\sinh(s)\big) -\big(s^2/6\big) 
+ \big(s^4/180\big)\big], 
\end{split} \\[3mm]  
\zeta(7) &= -7 \pi^6 \, \int_0^\infty \, ds \, s^{-8} \big[F\big(s^2\big) - F_3\big(s^2\big)\big] \\[1mm]  
&= -7 \pi^6 \, \int_0^\infty \, ds \, s^{-8} \big[\ln \big(s^{-1}\sinh(s)\big) - \big(s^2/6\big) 
+ \big(s^4/180\big) - \big(s^6/2835\big)\big],   \no \\
& \text{etc.}    \no
\end{align}   
\end{theorem}
\begin{proof}
The computation
\begin{align} 
F'(t) - F_n'(t) &= \dfrac{1}{2t} \big[t^{1/2} \coth\big(t^{1/2}\big) - 1\big] 
- \dfrac{1}{2} \sum_{k=1}^n \, \dfrac{2^{2k} B_{2k}}{(2k)!} \, t^{k-1}   \no \\
&= \dfrac{1}{2t} \bigg[t^{1/2} \coth \big(t^{1/2}\big) - \sum_{k=0}^n \, \dfrac{2^{2k} B_{2k}}{(2k)!} \, t^k \bigg] \no \\ 
& = \dfrac{1}{2t} \big[t^{1/2} \coth \big(t^{1/2}\big) - t^{1/2} \coth_n \big(t^{1/2}\big)\big], \quad t \geq 0,  
\end{align} 
and \eqref{2.51} then show 
\begin{align}
\zeta(z) &= \sin(\pi z/2) \pi^{z-1} \int_0^{\infty} ds \, s^{-z} [\coth(s) - \coth_n(s)]    \no \\
&=\sin (\pi z/2)  \pi^{z-1}  \int_0^\infty \, dt \, t^{-z/2} [F'(t) - F_n'(t)]    \no \\
&= (z/2) \sin (\pi z/2) \pi^{z-1}  \int_0^\infty \, dt \, t^{-z/2-1} [F(t) - F_n(t)], \\ 
& \hspace*{3.56cm} \Re(z) \in (\max(1,2n), 2n+2),    \no
\end{align}
after an integration by parts. 
\end{proof}

\appendix
\section{Basic Formulas for the Riemann $\zeta$-Function} \lb{sA}
\renewcommand{\theequation}{A.\arabic{equation}}
\renewcommand{\thetheorem}{A.\arabic{theorem}}
\setcounter{theorem}{0} \setcounter{equation}{0}

We present a number of formulas for $\zeta(z)$ and special values of $\zeta(\, \cdot \,)$. It goes 
without saying that no such collection can ever attempt at any degree of completeness, and certainly 
our compilation of formulas is no exception in this context. 

{\bf Definition.}
\begin{align}
\zeta(z) &= \sum_{k \in \bbN} k^{-z}, \quad z \in \bbC, \, \Re(z) > 1     \lb{A.1}\\
&= \big[1 - 2^{-z}\big]^{-1} \sum_{k \in \bbN_0} (2k + 1)^{-z}, \quad \Re(z) > 1,
\quad \text{\cite[p.~19]{MOS66}}     \lb{A.2} \\
&= \big[1 - 2^{1-z}\big]^{-1} \sum_{k \in \bbN} (-1)^{k + 1} k^{-z}, \quad \Re(z) > 0,
\quad \text{\cite[p.~19]{MOS66}}.
\end{align}

{\bf Functional equation:}
\begin{align}
\zeta(z) = 2^z \pi^{z-1} \sin(\pi z/2) \Gamma(1 - z) \zeta(1 - z), \quad z \in \bbC, \, \Re(z) < 0.
\end{align}

{\bf Alternative formulas:}
\begin{align}
\zeta(z) &= \Gamma(z)^{-1} \int_0^{\infty} dt \, \f{t^{z - 1}}{e^t - 1}, \quad z \in \bbC, \, \Re(z) > 1  \\
&= \mu^{z} \Gamma(z)^{-1} \int_0^{\infty} dt \, \f{t^{z - 1}}{e^{\mu t} - 1}, \quad z \in \bbC, \, \Re(z) > 1
\; \Re(\mu) > 0, \quad \text{\cite[3.4111]{GR80}}  \\
&= \Gamma(z)^{-1} [1 - 2^{1 - z}]^{-1}
\int_0^{\infty} dt \, \f{t^{z - 1}}{e^{t} + 1}, \quad z \in \bbC, \, \Re(z) > 0    \\
&= \mu^{z} \Gamma(z)^{-1} [1 - 2^{1 - z}]^{-1}
\int_0^{\infty} dt \, \f{t^{z - 1}}{e^{\mu t} + 1}, \quad z \in \bbC, \, \Re(z) > 0 
\; \Re(\mu) > 0,     \no \\
& \hspace*{8.5cm}  \text{\cite[3.4113]{GR80}}, 
\end{align}
where
\begin{equation}
\Gamma(z) = \int_0^{\infty} dt \, t^{z - 1} e^{- t}, \quad z \in \bbC, \, \Re(z) > 0.
\end{equation}
In addition,
\begin{align} 
\zeta(x) &= \Gamma(x)^{-1} \int_0^1\int_0^1 ds dt \, \f{[\ln(st)]^{x - 2}}{1 - st}, \quad x > 3,
\quad \text{\cite{Wo2}} \\
&= e^{i \pi (1 - x)} \Gamma(x)^{-1} \int_0^1 dt \, \f{\ln(t)^{x - 1}}{1 - t}, \quad x > 1, 
\quad \text{Jensen (1895),} \quad \text{\cite[4.2714]{GR80}},
\end{align}
and 
\begin{align} 
\zeta(z) &= \pi^{z/2} \Gamma(z/2)^{-1} \int_0^{\infty} dt \, t^{(z/2) - 1} \sum_{k \in \bbN} e^{- k^2 \pi t}  \\
&=  \pi^{z/2} \Gamma(z/2)^{-1} \sum_{k \in \bbN} \int_0^{\infty} dt \, t^{(z/2) - 1} e^{- k^2 \pi t},
\quad z \in \bbC, \, \Re(z) > 1, \quad \text{\cite{Wi2}}    \\
&= \f{2^{z - 1}}{z - 1} - 2^z \int_0^{\infty} dt \, \f{\sin(z \arctan(t))}{(1 + t^2)^{z/2} (e^{\pi t} + 1)},
\quad z \in \bbC\backslash\{1\},  \\
& \hspace*{4.85cm }\text{\cite[9.5134]{GR80}}, \; \text{\cite[p.~21]{MOS66}}   \no \\
&= \f{2^{z - 1}}{\big[1 - 2^{1 - z}\big]} \int_0^{\infty} dt \, \f{\cos(z \arctan(t))}{(1 + t^2)^{z/2}
\cosh(\pi t/2)},  \quad z \in \bbC\backslash\{1\},  \\ 
& \hspace*{7.24cm}  \text{\cite[p.~21]{MOS66}}    \no \\
&= \f{1}{2} + \f{1}{z - 1} + 2 \int_0^{\infty} dt \, \f{\sin(z \arctan(t))}{(1 + t^2)^{z/2} (e^{2\pi t} - 1)},
\quad z \in \bbC\backslash\{1\},     \lb{A.16} \\ 
& \hspace*{3.4cm} \text{Jensen's formula (1895),}  \quad 
\text{\cite[p.~21]{MOS66}}   \no \\
&= a^z \f{2^{z-1}}{\big[2^z - 1\big]} \Gamma(z)^{-1} \int_0^{\infty} dt \, \f{t^{z - 1}}{\sinh(at)},
\quad \Re(z) > 1, \; a > 0, \quad \text{\cite[3.5231]{GR80}}    \lb{A.17} \\
&= \Gamma(z+1)^{-1} 4^{-1} (2a)^{z + 1} \int_0^{\infty} dt \, \f{t^{z}}{[\sinh(at)]^2}, \quad
\Re(z) > - 1, \; \Re(a) > 0,    \no \\
& \hspace*{8.5cm} \text{\cite[35271]{GR80}}   \\
&= \Gamma(z+1)^{-1} 4^{-1} (2a)^{z + 1} \big[1 - 2^{1 - z}\big]^{-1}
\int_0^{\infty} dt \, \f{t^{z}}{[\cosh(at)]^2}, \\
& \hspace*{1.9cm} \Re(z) > - 1, \, z \neq 1, \; \Re(a) > 0, \quad  \text{\cite[35273]{GR80}}   \no \\
&= \Gamma(z+1)^{-1} \big[2 - 2^{2 - z}\big]^{-1} \int_0^{\infty} dt\,
\f{t^z}{\cosh(t) + 1}, \quad \Re(z) > 0, \, z \neq 1,  \\
& \hspace*{8.3cm} \text{\cite[3.5316]{GR80}}  \no \\
&= 2^{-1} + \Gamma(z)^{-1} 2^{z-1}\int_0^{\infty} dt \, t^{z - 1} e^{-2t} \coth(t), \quad
\Re(z) > 1, \quad \text{\cite[3.5513]{GR80}}   \\&= 2^{-1} + \Gamma(z)^{-1} 2^{z-1}\int_0^{\infty} dt \, t^{z - 1} e^{-2t} \coth(t), \quad
\Re(z) > 1, \quad \text{\cite[3.5513]{GR80}}   \\
&= \Gamma(z)^{-1} 2^{z-1}\int_0^{\infty} dt \, t^{z - 1} \f{e^{-t}}{\sinh(t)}, \quad
\Re(z) > 1, \quad \text{\cite[3.5521]{GR80}}   \\
&= \Gamma(z)^{-1} 2^{z-1} \big[1 - 2^{1 - z}\big]^{-1}
\int_0^{\infty} dt \, t^{z - 1} \f{e^{-t}}{\cosh(t)}, \quad
\Re(z) > 0, \, z \neq 1, \no \\
& \hspace*{8.5cm} \text{\cite[3.5523]{GR80}}    \\
&= 2^{z} \Gamma(z)^{-1} \int_0^1 dt \, [\ln(1/t)]^{z - 1} \f{t}{1 - t^2}, \quad
\Re(z) > 0, \quad \text{\cite[4.27212]{GR80}}    \\
&= \Gamma(z + 1)^{-1} \int_0^{\infty} dt\, \f{t^z e^t}{\big[e^t - 1\big]^2}, \quad \Re(z) > 1,
\quad \text{\cite[p.~20]{MOS66}}  
\end{align}

\begin{align} 
\zeta(z) 
&= \Gamma(z + 1)^{-1} \big[1 - 2^{1 - z}\big]^{-1}
\int_0^{\infty} dt\, \f{t^z e^t}{\big[e^t + 1\big]^2}, \quad \Re(z) > 0,
\quad \text{\cite[p.~20]{MOS66}}   \\
& = 2 \sin(\pi z/2) \, \int_0^\infty dt \, \f{t^{-z}}{e^{2 \pi t} -1}, 
\quad \Re(z) < 0, \quad \text{\cite[p.~104]{Li47}}   \\
& = (2^z-1)^{-1} \dfrac{2^{z-1} z}{z-1} + 2 (2^z-1)^{-1} \int_0^\infty dt \, 
\dfrac{\sin(z \arctan(2t))}{[(1/4) + t^2]^{z/2}} \, \dfrac{1}{e^{2\pi t} - 1},     \lb{A.29} \\ 
& \hspace*{6.5cm} z \in \bbC \backslash\{1\}, \quad \text{\cite[p.~279]{WW86}}.   \no   
\end{align}

{\bf Specific values:}
\begin{align}
\zeta (2n) = \f{(-1)^{n + 1}(2 \pi)^{2n}B_{2n} }{2 (2n)!}, \quad n \in \bbN_0,      \lb{A.30a}
\end{align}
where  $B_m$ are the Bernoulli numbers generated, for instance, by
\begin{equation}
\f{w}{e^{w} - 1} = \sum_{m \in \bbN_0} B_m \f{w^m}{m!}, \quad w \in \bbC, \, |w| < 2 \pi,
\end{equation}
in particular,
\begin{align}
& B_0 = 1, \; B_1 = - 1/2, \; B_2 = 1/6, \; B_3 = 0, \; B_4 = - 1/30, \; B_5 = 0, \; B_6 = 1/42, \;
\text{etc.,}  \\
& B_{2k+1} = 0, \; k \in \bbN.     \lb{A.33} 
\end{align}
Moreover, one has the {\bf generating functions} for $\zeta(2n)$,
\begin{align}
- (\pi z/2) \cot(\pi z) &= \sum_{n \in \bbN_0} \zeta(2n) z^{2n}, \quad |z| < 1, \; \zeta(0) = -1/2,
\lb{A.29a} \\
 - (\pi z/2) \coth(\pi z) &= \sum_{n \in \bbN_0} (-1)^n \zeta(2n) z^{2n}, \quad |z| < 1, \; \zeta(0) = -1/2,
 \lb{A.30}
\end{align}
and \cite{Wi2}
\begin{equation}
(n!/6)[\zeta(n - 2) - 3 \zeta(n - 1) + 2 \zeta(n)] = \int_0^{\infty} dt \, \f{t^n e^t}{(e^t - 1)^4}, \quad
n \in \bbN, \, n \geq 4.
\end{equation}
Choosing $k = 2n$, $n \in \bbN$, even,  employing \eqref{A.30a} for $\zeta(2n), \zeta(2n - 2)$, yields
a formula for $\zeta(2n - 1)$. Moreover, 
\begin{align}
\zeta(2n+1) &= \f{1}{(2n)!} \int_0^{\infty} dt \, \f{t^{2n}}{e^t - 1}, \quad n \in \bbN \\
&= \f{(-1)^{n + 1} (2 \pi)^{2n + 1}}{2 (2n + 1)!} \int_0^1 dt \, B_{2n + 1}(t) \cot(\pi t), \quad n \in \bbN,
\quad \text{\cite{DM09}},
\end{align}
where $B_m(\, \cdot \,)$ are the Bernoulli polynomials,
\begin{equation}
B_m(z) = \sum_{j = 0}^m \begin{pmatrix} m \\ j \end{pmatrix} B_j z^{m - j}, \quad t \in \bbC,
\end{equation}
generated, for instance, by
\begin{equation}
\f{w e^{z w}}{e^{w} - 1} = \sum_{m \in \bbN_0} B_m(z) \f{w^m}{m!}, \quad w \in \bbC, \,  |w| < 2 \pi.
\end{equation}
Explicitly,
\begin{align} 
\begin{split} 
& B_0(x) = 1, \;  B_1(x) = x - (1/2), \; B_2(x) = x^2 - x + (1/6),    \\
& B_3(x) = x^3 - (3/2)x^2 + (1/2)x, \; \text{etc.,}     \lb{A.35} 
\end{split} \\
& B_n(0) = B_n, \; n \in \bbN, \quad B_1(1) = - B_1 = 1/2, \;
B_n(1) = B_n, \; n \in \bbN_0 \backslash \{1\},     \lb{A.36} \\
& B_n'(x) = n B_{n-1}(x), \quad n \in \bbN, \; x \in \bbR.    \lb{A.37}
\end{align}

In addition, for $n \in \bbN$,
\begin{align}
\zeta(2n+1) &= \f{a^2 (2a)^{2n}}{[2^{- 2n- 1} -1]} \f{1}{(2n+1)!} \int_0^{\infty} dt \, t^{2n + 1}
\f{\cosh(at)}{[\sinh(at)]^2}, \quad a \neq 0, \quad \text{\cite[3.5279]{GR80}}  \\
&= \f{2^{2n}}{\big[2^{2n} - 1\big]} [(2n)!]^{-1} \int_0^1 dt \, \f{[\ln(t)]^{2n}}{1 + t},
\quad \text{\cite[4.2711]{GR80}}    \\
&= \f{2^{2n + 1}}{\big[2^{2n + 1} - 1\big]} [(2n)!]^{-1} \int_0^1 dt \, \f{[\ln(t)]^{2n}}{1 - t^2},
\quad \text{\cite[4.2711]{GR80}},   \\
\zeta(n) &= [(n-1)!]^{-1} \int_0^1 dt \, \f{[\ln(1/t)]^{n-1}}{1 - t},  \quad \text{\cite[4.2729]{GR80}}.
\end{align}

Just for curiosity,
\begin{equation}
\zeta(3) = 1.2020569032 .....
\end{equation}
Apery \cite{Ap79} proved in 1978 that $\zeta(3)$ is irrational (see also Beukers \cite{Be79}, 
van der Poorten \cite{Po79}, Zudilin \cite{Zu02},
and \cite{Wi1}, \cite{Wo1}).

Moreover, 
\begin{align}
\zeta(3) &= \sum_{k \in \bbN} k^{-3} = \f{8}{7} \sum_{k\in\bbN_0} (2k + 1)^{-3}
= \f{4}{3} \sum_{k \in \bbN_0} (-1)^k (k + 1)^{-3}, \quad \text{\cite{Wi1}}   \\
&= \f{1}{2} \int_0^{\infty} dt \, \f{t^2}{e^t - 1}, \quad \text{\cite{Wi1}}    \\
&= \f{2}{3}  \int_0^{\infty} dt \, \f{t^2}{e^t + 1}, \quad \text{\cite{Wi1}}    \\
&= \f{4}{7} \int_0^{\pi/2} dt \, t \ln([1/\cos(t)] + \tan(t)), \quad \text{\cite{Wi1}}    \\
&= \f{8}{7}\bigg[\f{\pi^2 \ln(2)}{4} + 2 \int_0^{\pi/2} dt \, t \ln(\sin(t)) \bigg],  \quad \text{\cite{Wo1}} \\
&= - \f{1}{2} \int_0^1 \int_0^1 dxdy \, \f{\ln(xy)}{1 - xy}, \quad \text{\cite{Be79}}   \\
&= \int_0^1 \int_0^1 \int_0^1 dxdydz \, \f{1}{1 - xyz}, \quad \text{\cite{Wi1}}    \\
&= \pi \int_0^{\infty} dt \, \f{\cos(2 \arctan(t))}{(1+t^2) [\cosh(\pi t / 2)]^2}, \quad \text{\cite{Wi1}}    \\
&= \f{8 \pi^2}{7} \int_0^1 dt \, \f{t(t^4 - 4 t^2 + 1) \ln(\ln(1/t))}{(1 + t^2)^{4}}, \quad \text{\cite{Wi1}}  \\
&= \f{8 \pi^2}{7} \int_1^{\infty} dt \, \f{t(t^4 - 4 t^2 + 1) \ln(\ln(t))}{(1 + t^2)^{4}}, \quad \text{\cite{Wi1}}    
\end{align}

\begin{align} 
\zeta(3) &= 10 \int_0^{1/2} dt \, \f{[\arcsinh(t)]^2}{t} \quad \text{\cite[p.~46]{Fi03}}   \\
& = (2/7) \pi^2 \ln(2) + (4/7) \int_0^{\pi} dt \, t \ln(\sin(t/2)) \quad \text{\cite[p.~46]{Fi03}}   \\
&= (2/7) \pi^2 \ln(2) -(8/7) \int_0^1 dt \, \f{[\arcsin(t)]^2}{t} \quad \text{\cite[p.~46]{Fi03}}   \\
&= (2/7) \pi^2 \ln(2) -(8/7) \int_0^{\pi/2} dt \, t^2 \cot(t)   \quad \text{\cite[p.~46]{Fi03}}   \\
&= -\dfrac{2}{7} \pi^2 \ln(2) - \dfrac{16}{7} \, \int_0^1 dt \, \dfrac{\arcth(t) \ln (t)}{t(1-t^2)}  \lb{A.64} \\
&= - \dfrac{4}{3} \, \int_0^1 dt \, \dfrac{\ln(t) \ln(1+t)}{t}    \\
& = - 8 \, \int_0^1 dt \, \dfrac{\ln(t) \ln(1+t)}{1+t}    \\
&= \int_0^1 dt \, \dfrac{\ln(t) \ln(1-t)}{1-t} = \int_0^1 dt \, \dfrac{\ln(t) \ln(1-t)}{t}     \\
&= \dfrac{1}{4} \pi^2 \ln(2) + \int_0^1 dt \, \dfrac{\ln(t) \ln(1+t)}{1-t}      \\
&= \dfrac{2}{13} \pi^2 \ln(2) + \dfrac{8}{13} \, \int_0^1 dt \, \dfrac{\ln(t) \ln(1-t)}{1+t}     \\
&= \dfrac{2}{7} \, \int_0^{\pi/2} dt \, \dfrac{t (\pi - t)}{\sin(t)}.     \lb{A.69}
\end{align}
Formulas \eqref{A.64}--\eqref{A.69} were provided by Glasser and Ruehr and can be found in 
\cite[Problem 80-13]{Kl90}. Finally, we also recall, 
\begin{align}
\zeta(3) &= 1 + \int_0^\infty dt \, \dfrac{6t-2t^3}{(1+t^2)^3} \, \dfrac{1}{e^{2\pi t} - 1}, 
\quad \text{\cite[p.~274]{He88}}  \\
&= \dfrac{6}{7} + \dfrac{2}{7} \,\int_0^\infty dt \, \dfrac{\sin(3 \arctan(2t))}{[(1/4) + t^2]^{3/2}} 
\, \dfrac{1}{e^{2\pi t} - 1}    \lb{A.71} \\
&= \dfrac{6}{7} + \dfrac{8}{7} \,\int_0^\infty dt \, \dfrac{\sin(3 \arctan(t))}{(1 + t^2)^{3/2}} 
\, \dfrac{1}{e^{\pi t} - 1}     \lb{A.72} \\
&= 2 - 8 \,\int_0^\infty dt \, \dfrac{\sin(3 \arctan(t))}{(1 + t^2)^{3/2}} \, \dfrac{1}{e^{\pi t} + 1}     \lb{A.73} \\
&= 1 + 2 \,\int_0^\infty dt \, \dfrac{\sin(3 \arctan(t))}{(1 + t^2)^{3/2}} \, \dfrac{1}{e^{2\pi t} - 1}.   \lb{A.74}
\end{align}
Formulas \eqref{A.71}--\eqref{A.73} are due to Jensen (1895) and are special cases of results to be 
found in \cite[p.~279]{WW86} (cf.\ \eqref{A.16}, \eqref{A.29}); finally, \eqref{A.74} is a consequence of \eqref{A.72} and \eqref{A.73}. 

For more on $\zeta(3)$ see also \cite[p.~42--45]{Fi03}.

For a wealth of additional formulas, going beyond what is recorded in this appendix, we also refer to 
\cite{Mi13} and \cite{Se19}.

\noindent {\bf Acknowledgments.} 
We are indebted to the anonymous referee for kindly bringing references \cite{Mi13} and \cite{Se19} to our attention. 


\end{document}